\documentclass{amsart}

\usepackage{amscd,amsmath,amssymb,amsfonts,bbm, calligra, xspace}
\usepackage[T1]{fontenc}
\usepackage{lmodern}
\usepackage{mathtools}
\usepackage{leftindex}
\usepackage{enumerate}
\usepackage{multicol}
\usepackage[all]{xy}
\usepackage{mathrsfs}
\usepackage{footmisc}
\usepackage[unicode,pdfborder={0 0 0},final]{hyperref}
\usepackage{stackrel}

\newtheorem{thm}{Theorem}[section]
\newtheorem{prop}[thm]{Proposition}

\newtheorem{lem}[thm]{Lemma}
\newtheorem{cor}[thm]{Corollary}

\theoremstyle{definition}

\newtheorem{quest}[thm]{Question}

\theoremstyle{remark}
\newtheorem{rem}[thm]{Remark}
\newtheorem{rems}[thm]{Remarks}

\numberwithin{equation}{section}

\newcommand{\Ext}{\mathrm{Ext}}

\newcommand{\Coker}{\mathrm{Coker}}

\newcommand{\Ima}{\mathrm{Im}}
\newcommand{\Hom}{\mathrm{Hom}}

\newcommand{\wc}{\mathrm{wc}}

\newcommand{\pt}{\mathrm{pt}}

\newcommand{\U}{\mathrm{U}}
\newcommand{\SO}{\mathrm{SO}}

\newcommand{\Gal}{\mathrm{Gal}}

\newcommand{\RR}{\mathrm{R}}

\makeatletter
\def\myrightarrow{{\setbox\z@\hbox{$\rightarrow$}\dimen0\ht\z@\multiply\dimen0 6\divide\dimen0 10\ht\z@\dimen0\box\z@}}
\def\myrightarrowfill@{\arrowfill@\relbar\relbar\myrightarrow}
\newcommand{\myxrightarrow}[2][]{\ext@arrow 0359\myrightarrowfill@{#1}{#2}}
\makeatother

\def\loccit{\emph{loc}.\kern3pt \emph{cit}.{}\ }
\def\eg{e.g.\kern.3em}
\def\ie{i.e.,\ }
\def\resp {\text{resp.}\kern.3em}

\def\Z{\mathbb Z}
\def\C{\mathbb C}
\def\F{\mathbb F}
\def\G{\mathbb G}

\def\Q{\mathbb Q}

\def\R{\mathbb R}

\def\bS{\mathbb S}

\def\cO{\mathcal{O}}

\def\cS{\mathcal{S}}

\def\ci{\mathcal{C}^{\infty}}

\def\whK{\widehat{K}}

\begin{document}

\title[]{On the Artin vanishing theorem for Stein spaces}

\author{Olivier Benoist}
\address{D\'epartement de math\'ematiques et applications, \'Ecole normale sup\'erieure, CNRS,
45 rue d'Ulm, 75230 Paris Cedex 05, France}
\email{olivier.benoist@ens.fr}


\renewcommand{\abstractname}{Abstract}
\begin{abstract}
Artin vanishing theorems for Stein spaces refer to the vanishing of some of their (co)homology groups in degrees higher than the dimension.  We obtain new positive and negative results concerning Artin vanishing for the cohomology of a Stein space relative to a Runge open subset. We also prove an Artin vanishing theorem for the $\Gal(\C/\R)$\nobreakdash-equivariant cohomology of a $\Gal(\C/\R)$\nobreakdash-equivariant Stein space relative to the fixed locus.
\end{abstract}
\maketitle

\section*{Introduction}

\subsection{Artin vanishing theorems for Stein spaces}
\label{parArtin}

Stein spaces (see \cite{GRStein}) are the complex-analytic analogues of affine complex algebraic varieties. They are exactly those  complex spaces whose higher coherent cohomology groups all vanish.  Examples include all closed complex subspaces of $\C^N$ (for some $N\geq 0$). 

Let $S$ be a Stein space of dimension $n$.  Artin vanishing theorems for $S$ assert that some of the (co)homology groups of $S$ (with values in abelian groups,  local systems...) vanish in degrees $>n$. 
In this introductory paragraph, we review the known results in this direction. 
Since these results are purely topological, it is harmless to assume that $S$ is reduced.

\subsubsection{The homotopy type of a Stein space}
\label{Artin1}

The theme which we consider here goes back to the work of Andreotti and Frankel. They proved in \cite[\S 2]{AF} that if~$S$ is nonsingular, then it has the homotopy type of a CW complex of dimension~$n$. It follows that $H^k(S,A)=0$ and $H_k(S,A)=0$ for any $k>n$ and any abelian group~$A$, and that $H_n(S, \Z)$ is a free abelian group.

These statements were extended by Narasimhan \cite[Theorem 3]{Narasimhanhomo} 
and Hamm (\cite[Satz 1]{Hamm}, see also \cite[Korollar]{Hamm2}) to possibly singular Stein spaces. 
Their proofs are based on Morse theory.
Applied to affine complex algebraic varieties, these results are closely related to (and imply) the weak Lefschetz hyperplane theorem.

\subsubsection{Weakly constructible coefficients}
\label{Artin2}

The theorems of Andreotti--Frankel and Hamm were generalized in several directions. 
First, it is possible to consider cohomology with more general (sheaf) coefficients. 

 In the algebraic setting,  for \'etale cohomology, such an extension applicable to any torsion \'etale sheaf was provided by Artin \cite[XIV, Corollaire~3.2]{SGA43}.  
To state an analogue in Stein geometry,  we say that a sheaf $\F$ on $S$ is \textit{weakly constructible} if it is locally constant in restriction to all the strata of some analytic stratification of $S$.
It then follows from the stratified Morse theory of Goresky and MacPherson~\cite{GMP}
that $H^k(S,\F)=0$ for any $k>n$ and any weakly constructible sheaf~$\F$ on~$S$ (see Kashiwara--Schapira \cite[Theorem~10.3.8]{KS}
and Sch\" urmann \cite[Corollary~6.1.2]{Schurmann}).
Cohomology with values in weakly constructible sheaves encompasses in particular cohomology groups relative to a closed analytic subset (which had been considered earlier, see \eg \cite[Theorem 3]{CM}).

\subsubsection{\texorpdfstring{$\cO(S)$}{O(S)}-convex subsets}
\label{Artin3}

In a second direction, it is possible to consider relative (co)homology groups of $S$ with respect to  subsets of $S$ with appropriate holomorphic convexity properties.   We note that these generalizations actually play an important role in the proofs of the results stated in \S\ref{Artin1} and \S\ref{Artin2}.

Recall that the $\cO(S)$\nobreakdash\textit{-con\-vex hull} of a compact subset $K\subset S$ is the compact subset $\whK_{\cO(S)}=\{s\in S\mid |f(s)|\leq \sup_{t\in K}|f(t)|\textrm{ for all }f\in\cO(S)\}$, and that a compact subset~$K\subset S$ is $\cO(S)$\textit{-convex} if $K=\whK_{\cO(S)}$. 
It is then true that $H^k(S,K,\F)=0$  for any $k>n$, any $\cO(S)$-convex compact subset $K$ of $S$, and any weakly constructible sheaf $\F$ on $S$ (this is very close to \cite[Corollary 6.1.2]{Schurmann}, see Proposition \ref{propArtin} below for a proof of this precise statement).

Results concerning the relative (co)homology of $S$ with respect to suitably holomorphically convex open subsets have appeared earlier, but turn out to be more subtle.
An open subset~$U\subset S$ is said to be \textit{Runge} if for all compact subsets~$K\subset U$, one has $\whK_{\cO(S)}\subset U$.  It is equivalent to require that $U$ is Stein and that the restriction map $\cO(S)\to \cO(U)$ has dense image (see \cite[VII, A, Corollary~9]{GunningRossi}). 
 It is then true that $H_k(S,U,A)=0$ for any $k>n$, any Runge open subset $U$ of~$S$, and any abelian group $A$ (this is due to Andreotti and Narasimhan \cite[Theorem~1]{AN} when $S$ has at most isolated singularities, and the general case was proved by Col\textcommabelow{t}oiu and Mihalache \cite[Theorem~1]{CM}).

\subsection{Relative cohomology of Runge pairs}
\label{parintroRunge}

Let $U\subset S$ be a Runge open subset. Combined with the universal coefficient theorem, the  Andreotti--Narasimhan and Col\textcommabelow{t}oiu--Mihalache theorems stated at the end of \S\ref{Artin3} show that $H^k(S,U,A)=0$ for any $k>n+1$ and any abelian group $A$. This naturally raises the question of what happens when $k=n+1$ (and, say,  $S$ is nonsingular and~$A=\Z$). 
 
\begin{quest}
\label{qRunge}
Let $S$ be a Stein manifold of dimension $n$ and let $U\subset S$ be a Runge open subset.  Is it true that $H^{n+1}(S,U,\Z)=0$?
\end{quest}

Applying the universal coefficient theorem as above shows that Question \ref{qRunge} has a positive answer when $H_n(S, U, \Z)$ is a free abelian group.

When $n=1$, Question \ref{qRunge} is stated explicitely as an open problem by Col\textcommabelow{t}oiu in \cite[Problem 4]{Coltoiupb}. When $n$ is arbitrary,  it was also raised by Scholze in \cite{Scholze}.  
The first goal of this article is to answer this question. 

Say that a connected orientable~$\ci$ surface~$S$ has \textit{finite genus} if it can be exhausted by connected compact $\ci$ subsurfaces with boundary $(S_i)_{i\geq 1}$ whose genera $g(S_i)=1-\frac{\chi(S_i)+b_0(\partial S_i)}{2}$ (where $\chi$ is the Euler characteristic and $b_0$ denotes the number of connected components) are bounded.
By Richards' classification of possibly noncompact $\ci$ surfaces \cite[Theorem 3]{Richards}, it is equivalent to require that~$S$ can be realized as an open subset of a connected compact orientable $\ci$ surface.

Our main two results concerning Question \ref{qRunge} in dimension $1$ are the following.

\begin{thm}[Corollary \ref{corcoho}]
\label{th1a}
Let $S$ be a connected Stein manifold of dimension~$1$ that has finite genus. Let $U\subset S$ be a Runge open subset. Then $H^k(S,U,A)=0$ for any~$k>1$ and any abelian group~$A$.
\end{thm}

\begin{thm}[Corollary \ref{ex1}]
\label{th1b}
There exist a connected Stein manifold~$S$ of dimension $1$ and a Runge open subset $U\subset S$ such that $H^2(S,U,\Z)\neq 0$.
\end{thm}

From a $1$-dimensional counterexample to Question \ref{qRunge} given by Theorem \ref{th1b}, one gets an $n$-dimensional one by multiplying both $S$ and $U$ with $(\C^*)^{n-1}$ and using the K\"unneth formula (see \eg \cite[Theorem 3.16]{Hatcher}). We therefore obtain the next corollary.

\begin{cor}
\label{higherdim}
Fix $n\geq 1$. There exist a Stein manifold $S$ of dimension $n$ and a Runge open subset $U\subset S$ such that $H^{n+1}(S,U,\Z)\neq 0$.
\end{cor}

As pointed out by Col\textcommabelow{t}oiu in \cite[Problem 4]{Coltoiupb},  
Question \ref{qRunge} is of topological nature when $n=1$. As a consequence, so are our proofs of Theorems~\ref{th1a} and~\ref{th1b}.

The purely topological statement from which Theorem \ref{th1a} derives is the freeness of the abelian group $H_1(S,U,\Z)$ for any connected orientable $\ci$ surface~$S$ of finite genus and any open subset~$U$ of~$S$ (see Theorem \ref{thdim1a} below).  This statement is more delicate than the freeness of $H_1(S,\Z)$ itself (a consequence of the existence of a triangulation of~$S$,  for which the finite genus hypothesis is not required).  Our proof makes use of N\"obeling's freeness criterion for abelian groups \cite{Nobeling}.

 As for Theorem \ref{th1b}, we let $S$ be homeomorphic to the simplest connected orientable surface of infinite genus (the one which has only one end), and we choose~$U$ to be a small tubular neighbourhood of a (non locally finite; see Remark \ref{rems1} (iii)) disjoint union of simple closed curves on $S$.  Our construction is
based on the complete understanding, due to Meeks and Patrusky \cite{MP},  of which homology classes of a compact $\ci$ surface with boundary can be realized by a simple closed curve.

\subsection{\texorpdfstring{$\Gal(\C/\R)$}{Gal(C/R)}-equivariant Artin vanishing for Stein spaces}

In applications to real-analytic or real algebraic geometry, it is important to work with complex spaces $S$ equipped with an action of the 
Galois group ${G:=\Gal(\C/\R)\simeq\Z/2}$ such that the complex conjugation $\sigma\in G$ acts $\C$-antilinearly on the structure sheaf~$\cO_S$. We call them $G$-\textit{equivariant complex spaces} (see \eg \cite[Appendix~A]{Tight}). 
Examples include complex manifolds endowed with an antiholomorphic involution.

In this setting, the natural cohomology groups to consider are $G$-equivariant sheaf cohomology groups $H^k_G(S,\F)$ (in the sense of \cite[\S 5.2]{Tohoku}) with values in a weakly con\-structible $G$\nobreakdash-equivariant sheaf $\F$ (such as the constant $G$-equivariant sheaf associated with a $G$-module).

When $S^G\neq\varnothing$, one cannot expect any vanishing result for these groups, already for $\F=\Z/2$. Indeed the pull-back morphism $\Z/2\simeq H^k_G(\pt,\Z/2)\to H^k_G(S,\Z/2)$ admits a retraction given by restriction to a $G$-fixed point, and hence is nonzero. To extend to $G$-equivariant Stein geometry the Artin vanishing theorems considered in~\S\ref{parArtin}, one must instead consider $G$\nobreakdash-equivariant cohomology groups of a $G$\nobreakdash-equivariant Stein space $S$ relative to the fixed locus $S^G$. 

 In the algebraic setting and with torsion coefficients, such a theorem was obtained by Scheiderer in \cite[\S 18]{Scheiderer} (see especially \cite[Theorem 18.2, Corollary~18.11]{Scheiderer}).
The second aim of this article is to extend Scheiderer's theorem to the $G$-equivariant Stein setting, with possibly nontorsion weakly constructible coefficients, and relative to an arbitrary $G$-invariant~$\cO(S)$\nobreakdash-convex compact subset of~$S$. This yields $G$\nobreakdash-equivariant analogues of the cohomological statements considered in~\S\ref{parArtin} (thereby generalizing them, see Remark~\ref{remArtinG}~(i)).

\begin{thm}[Theorem \ref{thArtin}]
\label{th3}
Let $S$ be a $G$-equivariant Stein space of dimension $n$. Let $K\subset S$ be a $G$-invariant~$\cO(S)$\nobreakdash-convex compact subset. Let $\F$ be a $G$\nobreakdash-equivariant weakly constructible sheaf on $S$. Then $H^k_G(S,S^G\cup K,\F)=0$ for $k>n$.
\end{thm}

In the algebraic case and with torsion coefficients, Scheiderer's proof extends the arguments of Artin in \cite[XIV, \S 4]{SGA43} and proceeds by d\'evissage and induction on the dimension. In contrast,  our proof is ultimately based on (non-$G$-equivariant) stratified Morse theory.  To handle the contribution of $S^G$, we heavily rely on the flexibility arising from the possibility of changing the $\cO(S)$-compact subset $K$.

Let us comment on the relation between Theorem \ref{th3} and the results of \S\ref{parintroRunge}.  When proving Artin-type vanishing theorems for relative cohomology groups, the articles~\cite{AN} and \cite{CM} highlight the role of Runge open subsets.  Following their point of view,  one could try to prove Theorem \ref{th3} by showing that $S^G\cup K$ admits a basis of Runge open neighborhoods (which can be deduced from \cite[Corollary~3.2]{MWO} and Lemma \ref{lemCarleman}),  and by eventually applying a vanishing theorem for the relative cohomology of Runge pairs.  This line of reasoning led us to Question~\ref{qRunge}. The counterexamples given by Corollary \ref{higherdim} compelled us to find another approach,  emphasizing more the role of $\cO(S)$-convex compact subsets, and relying more heavily on particular properties of the subset $S^G$ of $S$ (see Lemma~\ref{lemCarleman}).

Applications of Theorem \ref{th3} to the arithmetic of fields of meromorphic functions on (possibly $G$-equivariant) Stein spaces, which were our main motivation for the present work,  will appear in subsequent articles (see \cite{Steinsurface}).

\section{Relative cohomology of Runge pairs}

In this section, we prove Theorem \ref{th1a} (in \S\ref{11}) and Theorem \ref{th1b} (in \S\ref{12}).
All homology groups may be thought of as singular homology groups, but cohomology groups are always understood to be sheaf cohomology groups (which may not coincide with singular cohomology groups).

\subsection{Surfaces of finite genus}
\label{11}

We start with a lemma.

\begin{lem}
\label{lemhomorel}
Let $\Sigma$ be a compact orientable $\ci$ surface. Let $S\subset\Sigma$ be an open subset. Then
\begin{enumerate}[(i)]
\item the group $H_2(\Sigma,S,\Z)$ identifies with the sheaf cohomology group $H^0(\Sigma\setminus S,\Z)$;
\item the group $H_1(\Sigma,S,\Z)$ is torsion-free.
\end{enumerate}
\end{lem}

\begin{proof}
Write $S$ as an increasing union of compact $\ci$ subsurfaces with boundary~$(S_i)_{i\geq 1}$. Define $\mathring{S}_i:=S_i\setminus \partial S_i$.  To prove (i), one then computes
\begin{equation}
\label{Lefschetz}
\begin{alignedat}{3}
H_2(\Sigma,S,\Z)&=\varinjlim_{i\geq 1}H_2(\Sigma, S_i,\Z)=\varinjlim_{i\geq 1}H_2(\Sigma\setminus \mathring{S}_i, \partial S_i,\Z)\\
&=\varinjlim_{i\geq 1}H^0(\Sigma\setminus \mathring{S}_i,\Z)=H^0(\Sigma\setminus S,\Z).
\end{alignedat}
\end{equation}
The first equality in (\ref{Lefschetz}) holds since the singular complex of the pair $(\Sigma,S)$ is the inductive limit of the singular complexes of the pairs $(\Sigma, S_i)$.  The second equality follows from excision,  and the third from Lefschetz duality. In turn, the fourth equality  of (\ref{Lefschetz}) results from \cite[Lemma 14.2]{Bredon}
as $\Sigma\setminus S=\varprojlim_{i\geq 1} \Sigma\setminus \mathring{S}_i$.

Let $m\geq 1$ be an integer. Consider the exact sequence
\begin{equation}
\label{multm}
H_2(\Sigma,S,\Z)\to H_2(\Sigma,S,\Z/m)\to H_1(\Sigma,S,\Z)\xrightarrow{m} H_1(\Sigma,S,\Z).
\end{equation}
In view of (i), the left arrow of (\ref{multm}) identifies with the reduction modulo $m$ morphism $H^0(\Sigma\setminus S,\Z)\to H^0(\Sigma\setminus S,\Z/m)$. As this morphism is surjective, the right arrow of (\ref{multm}) is injective. Consequently, the group $H_1(\Sigma,S,\Z)$ has no $m$-torsion.  Assertion (ii) follows since $m$ is arbitrary.
\end{proof}

\begin{rem}
Lemma \ref{lemhomorel} (i),  and in particular the role sheaf cohomology (as opposed to singular cohomology) plays in it, was inspired to us by Sitnikov's generalization of the Alexander duality theorem (\cite{Sitnikov}, 
see also \cite[bottom of p.~80]{Milnor}).
\end{rem}

\begin{thm}
\label{thdim1a}
Let $S$ be a connected orientable $\ci$ surface of finite genus.  Let~$U$ be an open subset of $S$. Then $H_1(S,U,\Z)$ is a free abelian group.
\end{thm}

\begin{proof}
It follows from Richards' classification of possibly noncompact $\ci$ surfaces (see \cite[Theorem 3]{Richards}) that $S$ is obtained from an open subset of the sphere $\bS^2$ by attaching finitely many handles, and hence that $S$ may be realized as an open subset of a connected compact orientable $\ci$ surface $\Sigma$.
Consider the following commutative diagram of relative homology exact sequences:
\begin{equation}
\label{highg}
\begin{aligned}
\xymatrix@C=1.5em@R=3ex{
H_2(\Sigma,\Z)\ar@{=}[d]\ar^{}[r]&H_2(\Sigma,U,\Z)\ar^{c}[d]\ar^{}[r]&H_1(U,\Z)\ar^{d}[d]\ar^{a}[r]&H_1(\Sigma,\Z)\ar@{=}[d]\ar^{}[r]&H_1(\Sigma,U,\Z)\ar^{}[d] \\
H_2(\Sigma,\Z)\ar^{}[r]&H_2(\Sigma,S,\Z)\ar^{}[r]&H_1(S,\Z) \ar^{b}[r]&H_1(\Sigma,\Z)\ar^{}[r]&H_1(\Sigma,S,\Z)\rlap{.}
}
\end{aligned}
\end{equation}
A diagram chase in (\ref{highg}) yields a short exact sequence
\begin{equation}
\label{extension}
0\to \Coker(c)\to\Coker(d)\to \Ima(b)/\Ima(a)\to 0.
\end{equation}

Since $H_1(\Sigma,U,\Z)$ and $H_1(\Sigma,S,\Z)$ are torsion-free by Lemma \ref{lemhomorel} (ii), the subgroups $\Ima(a)$ and $\Ima(b)$ of the finitely generated free abelian group $H_1(\Sigma,\Z)$ are saturated. We deduce that the abelian group~$\Ima(b)/\Ima(a)$ is free (of finite type).

Applying Lemma \ref{lemhomorel} (i) to the open subsets $S$ and $U$ of $\Sigma$ yields an isomorphism 
\begin{equation}
\label{Cokerc}
\Coker(c)=\Coker\big[H^0(\Sigma\setminus U,\Z)\to H^0(\Sigma\setminus S,\Z)\big].
\end{equation}
Let $A$ be the abelian group of bounded functions $f:\Sigma\setminus S\to \Z$.  Following N\"obeling (\cite{Nobeling}, see also \cite[\S 97]{Fuchs2}), we say that a subgroup $B$ of $A$ is \textit{Specker} if every~$f\in A$ can be written as a finite sum $f=\sum_i m_if_i$, where $m_i\in\Z$ and~$f_i$ is a characteristic function of a subset of $\Sigma\setminus S$.  Since $\Sigma\setminus S$ is compact,  the subgroup~$H^{0}(\Sigma\setminus S,\Z)$ of $A$ consisting of locally constant functions is Specker.  So is the image of the restriction map $H^{0}(\Sigma\setminus U,\Z)\to H^{0}(\Sigma\setminus S,\Z)\subset A$.  The abelian group $\Coker(c)$ is a quotient of two Specker subgroups of $A$ (see (\ref{Cokerc})), and hence is free by N\"obeling's theorem (\cite[Satz 2]{Nobeling}, see also \cite[Theorem 97.3]{Fuchs2}).

We now deduce from (\ref{extension}) that the abelian group $\Coker(d)$ is free as an extension of two free abelian groups. The long exact sequence of homology of the pair $(S,U)$ gives rise to an exact sequence
\begin{equation}
\label{lastes}
0\to\Coker(d)\to H_1(S,U,\Z)\to H_0(U,\Z).
\end{equation}
As the abelian group $H_0(U,\Z)$ is free, so is the image of the right arrow of (\ref{lastes}) (see \cite[Theorem~14.5]{Fuchs1}). It follows from (\ref{lastes}) that the abelian group $H_1(S,U,\Z)$ is free as an extension of two free abelian groups. 
\end{proof}

\begin{cor}
\label{corcoho}
Let $S$ be a connected Stein manifold of dimension $1$ that has finite genus. Let $U\subset S$ be a Runge open subset.  
Then $H^k(S,U,A)=0$ for any $k>1$ and any abelian group $A$.
\end{cor}

\begin{proof}
The universal coefficient theorem provides us with a short exact sequence
$$0\to \Ext^1(H_{k-1}(S,U,\Z),A)\to H^k(S,U,A)\to\Hom(H_k(S,U,\Z),A)\to 0.$$
As $S$ is Stein of dimension $1$ and $U$ is Runge,  one has $H_k(S,U,\Z)=0$ for~$k\geq 2$ (see \eg \cite[Theorem 1]{CM}). In addition, the group $\Ext^1(H_{1}(S,U,\Z),A)$ vanishes 
as $H_{1}(S,U,\Z)$ is a free abelian group by Theorem \ref{thdim1a}. The corollary follows.
\end{proof}

\subsection{A surface of infinite genus}
\label{12}

\begin{thm}
\label{thdim1b}
There exists a noncompact connected orientable $\ci$ surface $S$ and an open subset $U\subset S$ such that the morphism $H_1(U,\Z)\to H_1(S,\Z)$ is injective with non-free cokernel.
\end{thm}

\begin{proof}
Let $S_0$ be the torus $\R^2/\Z^2$ with a small open disc removed.  Define $\gamma_0:=\partial S_0$. For each $i\geq 1$, let $S_i$ be a copy of the torus $\R^2/\Z^2$ with two small open discs removed.  Let $\gamma_i$ and $\gamma'_i$ be the two connected components of $\partial S_i$.  Let $S$ be a connected orientable $\ci$ surface obtained from the disjoint union of the $(S_i)_{i\geq 0}$ by identifying $\gamma_i$ with $\gamma'_{i+1}$ for $i\geq 0$.
Let $S_{\leq j}$ be the subset $\cup_{i\leq j}S_i$ of $S$. It is a compact $\ci$ subsurface of $S$ with boundary $\partial S_{\leq j}=\gamma_j$.
Let $\alpha_i, \beta_i\in H_1(S_i,\Z)$ be the classes of a meridian and of an equator of the punctured torus $S_i$. 
The group~$H_1(S_{\leq i},\Z)$ is then freely generated by $\alpha_0,\beta_0,\dots,\alpha_i,\beta_i$.

We now construct by induction on $i\geq 1$ an oriented simple closed $\ci$ curve $\lambda_i$ in~$S_{\leq i}\setminus \partial S_{\leq i}$ with homology class $\alpha_{i-1}-2\alpha_i$ in $H_1(S_{\leq i},\Z)$, as well as a closed tubular neighborhood $T_i$ of $\lambda_i$ in $S_{\leq i}\setminus \partial S_{\leq i}$, with the property that $T_i\cap T_j=\varnothing$ if~$j<i$. Assume that $\lambda_1,\dots,\lambda_{i-1}$ (as well as $T_1,\dots, T_{i-1}$) have been constructed.

Consider the subgroup $\Lambda_i:=\langle \alpha_{j-1}-2\alpha_j\rangle_{1\leq j<i}$ of $H_1(S_{\leq i},\Z)$ and the compact~$\ci$ surface with boundary $A_i:=S_{\leq i}\setminus\cup_{j<i}\,\mathring{T}_j$. The Mayer--Vietoris long exact sequence associated with the covering of $S_{\leq i}$ by $A_i$ and~$\cup_{j<i}T_j$ shows that~$A_i$ is connected and yields an exact sequence
\begin{equation}
\label{MV}
0\to \langle\delta_j-\delta_j'\rangle_{1\leq j<i}\to H_1(A_i,\Z)\to H_1(S_{\leq i},\Z),
\end{equation}
where $\delta_j$ and $\delta_j'$ are the boundary components of $T_j$, and such that the image of the right arrow of (\ref{MV}) is the orthogonal $\Lambda_i^{\perp}$ of $\Lambda_i$ (with respect to the intersection product).
Let $\varepsilon_i\in H_1(A_i,\Z)$ be any lift of $\alpha_{i-1}-2\alpha_i\in \Lambda_i^{\perp}\subset H_1(S_{\leq i},\Z)$ in (\ref{MV}).

Let $B_i$ be the compact $\ci$ surface obtained by capping off~$A_i$ by gluing a disk to each of its boundary components.  We deduce from (\ref{MV}) that $H_1(B_i,\Z)$ identifies with the quotient 
$\Lambda_i^{\perp}/\Lambda_i$ of $H_1(A_i,\Z)$.
The class $\alpha_{i-1}-2\alpha_i$ is nonzero and primitive in $H_1(B_i,\Z)=\Lambda_i^{\perp}/\Lambda_i$. It therefore follows from \cite[Theorem 1]{MP} that there exists an oriented simple closed $\ci$ curve $\lambda_i$ on $A_i$ (which we may suppose disjoint from $\partial A_i$)  with homology class $\varepsilon_i$ in $H_1(A_i,\Z)$.  Its homology class in $H_1(S_{\leq i},\Z)$ is thus equal to $\alpha_{i-1}-2\alpha_i$. Letting $T_i$ be any small enough closed tubular neighborhood of $\lambda_i$ completes the induction.

To conclude the proof,  let $U$ be the disjoint union of the~$(T_i\setminus\partial T_i)_{i\geq 1}$.  It is an open subset of $S$.  As $H_1(U,\Z)$ is freely generated by the homology classes of the~$(\lambda_i)_{i\geq 1}$, the morphism~$H_1(U,\Z)\to H_1(S,\Z)$ is injective with image $\langle \alpha_{i-1}-2\alpha_i\rangle_{i\geq 1}$.  It follows that the image of $\alpha_0$ in $H_1(S,\Z)/H_1(U,\Z)$ is nonzero but divisible by $2^i$ for all $i\geq 1$. We deduce that the abelian group $H_1(S,\Z)/H_1(U,\Z)$ is not free.
\end{proof}

\begin{rems}
\label{rems1}
(i) A surface $S$ as in Theorem \ref{thdim1b} must have infinite genus (see Theorem \ref{thdim1a}). The surface $S$ considered in the proof of Theorem \ref{thdim1b} is therefore as simple as possible (it is the only connected orientable $\ci$ surface of infinite genus with only one end, see \cite[Theorem 1]{Richards}).

(ii) The open subset $U$ in Theorem \ref{thdim1b} must also necessarily be quite complicated.  For instance, it cannot be the complement of a closed $\ci$ subsurface with boundary~$S'$ of~$S$. 
To see it, let $T$ be the $\ci$ surface with boundary obtained from~$S'$ by capping off all the compact boundary components of $S'$ by gluing a union~$D$ of disks. Excision yields $H_1(S,U,\Z)=H_1(S',\partial S',\Z)=H_1(T,D\cup\partial T,\Z)$. The long exact sequence of homology of the pair $(T, D\cup\partial T)$ now realizes 
$H_1(S,U,\Z)$ as an extension of a subgroup of $H_0(D\cup \partial T,\Z)$ by $H_1(T,\Z)$, which are both free. It follows that $H_1(S,U,\Z)$, hence also its subgroup $H_1(S,\Z)/H_1(U,\Z)$, are free.

(iii) The main tool of the proof of Theorem \ref{thdim1b} is the construction of simple closed curves realizing appropriate homology classes on compact $\ci$ surfaces with boundary \cite[Theorem 1]{MP}.  The proof of this result in \cite{MP} is algorithmic.  The algorithm there is sufficiently complicated that the resulting curves may be quite intricate (see \eg the figures in \cite[pp.~263-264]{MP}).  A consequence is that the system of simple closed curves $(\lambda_i)_{i\geq 1}$ on $S$ that we construct in this way is not locally finite (and cannot be in view of (ii) above).
\end{rems}

\begin{cor}
\label{ex1}
There exist a Stein manifold $S$ of dimension $1$ and a Runge open subset $U\subset S$ such that $H^2(S,U,\Z)\neq 0$.
\end{cor}

\begin{proof}
Let $S$ and $U$ be as in Theorem \ref{thdim1b}. Fix a Riemannian metric on the~$\ci$ surface~$S$.  Since~$\SO(2)=\U(1)$, this determines an almost complex structure on~$S$. This almost complex structure is integrable (see \eg \cite[Theorem 4.2.6]{MDS}). Since~$S$ is noncompact, the resulting complex manifold is Stein by the Behnke--Stein theorem (\cite{BS}, see also \cite[Theorem 3.10.13]{Narasimhanbook}).  In addition, as the morphism~$H_1(U,\Z)\to H_1(S,\Z)$ is injective, 
the open subset $U$ of $S$ is Runge by the Behnke--Stein approximation theorem \cite[Satz 6]{BS}.

Since the cokernel of the morphism $H_1(U,\Z)\to H_1(S,\Z)$ is not free, the abelian group $H_1(S,U,\Z)$ is not free either (by \cite[Theorem~14.5]{Fuchs1}). As $H_1(S,U,\Z)$ is countable, we deduce from  a theorem of Stein (\cite[\S 9, Hilfssatz 3]{Steinab},  see also \cite[Theorem~99.1]{Fuchs2}) that the group $\Ext^1(H_1(S,U,\Z),\Z)$ is nonzero. The universal coefficient theorem then implies that $H^2(S,U,\Z)$ is nonzero either.
\end{proof}

\section{Artin vanishing for \texorpdfstring{$G$}{G}-equivariant Stein spaces}

The goal of this section is the $G$-equivariant version of the Artin vanishing theorem for Stein spaces stated as Theorem \ref{th3} in the introduction.

\subsection{Artin vanishing relative to an  \texorpdfstring{$\cO(S)$}{O(S)}-convex compact subset}
\label{parconstructible}

We first prove a version of Artin vanishing relative to an arbitrary $\cO(S)$-convex compact subset (see Proposition \ref{propArtin}).
We refer to \cite[\S 2]{NarasimhanLeviI} for the definition of a $\ci$ strongly plurisubharmonic (psh) function on a possibly singular
complex space.

\begin{lem}
\label{lempsh}
Let $S$ be a Stein space. Let $K\subset S$ be an $\cO(S)$-convex compact subset. Let $U\subset S$ be an open neighborhood of $K$. Then there exists a $\ci$ strongly psh exhaustion function ${\rho:S\to\R}$ such that $\rho<0$ on $K$ and $\rho>0$ on $S\setminus U$.
\end{lem}

\begin{proof}
When $S$ is nonsingular, this is exactly \cite[Theorem 5.1.6]{Hormander}. As noted in \cite[Proposition 2.5.1]{Forstneric}, the proof given in \loccit works in general (replacing \textit{forming a local system coordinates at $s$} by \textit{generating the maximal ideal of $\cO_{S,s}$}).
\end{proof}

\begin{prop}
\label{propArtin}
Let $S$ be a Stein space of dimension $n$. Let $K\subset S$ be an $\cO(S)$\nobreakdash-con\-vex compact subset. Let $\F$ be a weakly constructible sheaf on $S$. Then ${H^k(S, K,\F)=0}$ for $k>n$.
\end{prop}

\begin{proof}
Fix $k>n$.  Let $\cS$ be a Whitney stratification of~$S$ adapted to $\F$, \ie such that $\F$ is locally constant on the strata (see \eg \cite[Th\'eor\`eme 2.2]{Verdier}).

Assume first that ${K=\{s\in S\mid \rho(s)\leq \rho_0\}}$, where~${\rho:S\to\R}$ is a~$\ci$ strongly psh exhaustion function and~$\rho_0$ is a regular value of~$\rho$ with respect to~$\cS$.
Choose an increasing sequence
$(\rho_i)_{i\geq 0}$ of regular values of $\rho$ with respect to~$\cS$ 
tending to~$+\infty$.
Set~$K_i:=\{s\in S\mid \rho(s)\leq \rho_i\}$. 
 For~$i\geq 1$, stratified Morse theory shows that $H^k(K_{i},K_{i-1},\F)=H^k(K_{i},K,\F)=0$ (apply \cite[Theorem~3.64~(1)]{Constructible} 
with~$q=0$, noting that $\F\in\leftindex^p{D}_{\wc}^{\leq n}(S)$ in view of \cite[Definitions~2.16 and~2.20, Remark~2.17~(a)]{Constructible}, where the triangulated category~${D}^b_{\wc}(S)$ of weakly constructible bounded complexes of sheaves on~$S$ is endowed with its perverse $t$-structure for the middle perversity). The first vanishing implies that the projective system $(H^{k-1}(K_i,K,\F))_{i\geq 0}$ has surjective transition maps and hence satisfies the Mittag--Leffler condition.  It follows that~$H^k(S,K,\F)=\varprojlim_i H^k(K_i,K,\F)=0$.

Now let $K$ be arbitrary.  Using Lemma \ref{lempsh},  one can find a decreasing sequence~$(L_i)_{i\geq 1}$ of compact neighborhoods of $K$ with $\cap_{i\geq 1}L_i=K$, such that $L_i$ 
is a sublevel set of a $\ci$ strongly psh exhaustion function of $\rho_i:S\to\R$ (associated with a regular value with respect to $\cS$).  By the previous paragraph, one has $H^k(S,L_i,\F)=0$ for~$i\geq 1$.  After taking the direct limit of the exact sequences
$$H^{k-1}(L_i,K,\F)\to H^k(S,L_i,\F)\to H^k(S,K,\F)\to H^k(L_i,K,\F)$$
over all $i\geq 1$,
the two extreme terms vanish by \cite[Lemma 14.2]{Bredon}. 
We deduce an isomorphism
$H^k(S,K,\F)=\varinjlim_{i}H^k(S,L_i,\F)$. This group therefore vanishes.
\end{proof}

\subsection{Artin vanishing relative to the set of \texorpdfstring{$G$}{G}-fixed points}

In a second step,  in a $G$-equivariant setting, we prove an Artin vanishing theorem for cohomology relative to the set of $G$-fixed points (see Proposition \ref{propArtin2}).

\begin{lem}
\label{lemCarleman}
Let $S$ be a $G$-equivariant Stein space. Let $K\subset S$ be a $G$-invariant $\cO(S)$\nobreakdash-convex compact subset.  Let $L\subset S^G$ be a compact subset.  Then the compact subset $L\cup K$ of $S$ is $\cO(S)$-convex.
\end{lem}

\begin{proof}
Let $U$ be a relatively compact Runge neighborhood of $L\cup K$ in~$S$ (see \mbox{\cite[(1.1)]{NarasimhanLeviII}}). As the $\cO(S)$-convex hull and the $\cO(U)$-convex hull of $L\cup K$ coincide (see \cite[p.~919]{Narasimhan}), we may replace $S$ with $U$. Then
$S$ has finite embedding dimension.
Let $i:S\to\C^N$ be a holomorphic embedding (see \cite[Theorem~6]{Narasimhan}). Replacing $N$ with $2N$ and  $i$ with $s\mapsto(i(s)+\overline{i\circ\sigma(s)},\sqrt{-1}\cdot(i(s)-\overline{i\circ\sigma(s)}))$, we may assume that $i$ is $G$-equivariant.  Since $K$ is a $G$-invariant $\cO(\C^{N})$-convex compact subset of $\C^N$, we deduce from \cite[Theorem 2]{SC} that $L\cup K$ is $\cO(\C^{N})$-convex. As~$i^*:\cO(\C^{N})\to\cO(S)$ is onto,  the compact subset $L\cup K$ is also $\cO(S)$-convex.
\end{proof}

\begin{prop}
\label{propArtin2}
Let $S$ be a $G$-equivariant Stein space of dimension $n$. Let $K\subset S$ be a $G$-invariant~$\cO(S)$\nobreakdash-convex compact subset. Let $\F$ be a $G$\nobreakdash-equivariant weakly constructible sheaf on $S$. Then $H^k(S,S^G\cup K,\F)=0$ for $k>n$.
\end{prop}

\begin{proof}
Fix $k>n$.  Assume first that there is a $G$\nobreakdash-in\-vari\-ant~$\ci$ strongly psh exhaustion function $\rho:S\to\R$ with $K=\{s\in S\mid\rho(s)\leq 0\}$. For~$t\in\R$, set $S_{\leq t}:=\{s\in S\mid \rho(s)\leq t\}$ and $S_{< t}:=\{s\in S\mid \rho(s)<t\}$.
By \cite[Theorem 3]{NarasimhanLeviI},  for integers $0\leq j\leq i$ and $m\geq 1$, the open subset~$S_{<i+\frac{1}{m}}$ of~$S$ is Runge, and $S_{\leq j}$ is $\cO(S_{<i+\frac{1}{m}})$-convex (as it has a basis of Runge open neighborhoods).
Lemma \ref{lemCarleman} implies that $S_{\leq i}^G\cup S_{\leq j}$ is $\cO(S_{<i+\frac{1}{m}})$\nobreakdash-convex for~$m\geq 1$.  Proposition \ref{propArtin} then shows that $H^k(S_{<i+\frac{1}{m}},S_{\leq i}^G\cup S_{\leq j},\F)=0$ for $m\geq 1$.  Consequently, 
$$0=\varinjlim_{m\geq 1}H^k(S_{<i+\frac{1}{m}},S_{\leq i}^G\cup S_{\leq j},\F)=\varinjlim_{m\geq 1} H^k(S_{\leq i+\frac{1}{m}},S_{\leq i}^G\cup S_{\leq j},\F).$$
By \cite[Lemma 14.2]{Bredon}, this group coincides with $H^k(S_i,S_{\leq i}^G\cup S_{\leq j},\F)$, which therefore vanishes. The long exact sequences
$$\dots\to H^k(S_{\leq i},S_{\leq i}^G\cup S_{\leq j},\F)\to H^{k}(S_{\leq i},S_{\leq i}^G\cup K,\F)\to H^{k}(S_{\leq j},S_{\leq j}^G\cup K,\F)\to\dots$$
now show that the projective system $(H^{k-1}(S_{\leq i},S_{\leq i}^G\cup K,\F))_{i\geq 0}$ has surjective transition maps and hence satisfies the Mittag--Leffler condition.  
 It follows that 
$$H^k(S,S^G\cup K,\F)=\varprojlim_i H^k(S_{\leq i},S_{\leq i}^G\cup K,\F)= \varprojlim_i H^k(S_{\leq i},S_{\leq i}^G\cup S_{\leq 0},\F)=0.$$

Now let $K$ be arbitrary.  By Lemma \ref{lempsh},  there is a decreasing sequence~$(L_i)_{i\geq 1}$ of compact neighborhoods of $K$ with $\cap_{i\geq 1}L_i=K$, such that the~$L_i$ 
are sublevel sets of $\ci$ strongly psh exhaustion functions of $\rho_i: S\to\R$. After replacing~$\rho_i$ with~$\frac{\rho_i+\rho_i\circ\sigma}{2}$, we may assume that $L_i$ and $\rho_i$ are $G$-invariant.  
  By the previous paragraph, one has $H^k(S,S^G\cup L_i,\F)=0$ for~$i\geq 1$.  After taking the direct limit over all~$i\geq 1$ of the exact sequences
$$H^{k-1}(L_i,L_i^G\cup K,\F)\to H^k(S,S^G\cup L_i,\F)\to H^k(S,S^G\cup K,\F)\to H^k(L_i,L_i^G\cup K,\F),$$
the two extreme terms vanish by \cite[Lemma 14.2]{Bredon}. We deduce an isomorphism
$H^k(S,S^G\cup K,\F)=\varinjlim_{i}H^k(S,S^G\cup L_i,\F)$. This group therefore vanishes.
\end{proof}

\subsection{\texorpdfstring{$G$}{G}-equivariant Artin vanishing}

We finally deduce Theorem \ref{thArtin}, which concerns $G$-equivariant cohomology, from Proposition \ref{propArtin2} which deals with non-$G$\nobreakdash-equivariant cohomology. The principle of the argument is similar to the proof of \cite[Lemma 1.16]{BW1}.  We include the next lemma for lack of a convenient reference.

\begin{lem}
\label{lemvandim}
Let $X$ be a second-countable Hausdorff topological space. Assume that each point of $X$ admits a closed neighborhood homeomorphic to a closed semianalytic subset of dimension $\leq d$ of $\R^N$ for some $N\geq 0$.  Let $\Gamma$ be a finite group acting on~$X$.  Let $\G$ be a sheaf on~$X/\Gamma$. Then $H^k(X/\Gamma,\G)=0$ for $k>d$.
\end{lem}

\begin{proof} 
As $X$ is second-countable,  there is a countable covering $(X_i)_{i\in I}$ of~$X$, where each $X_i$ is a
closed semianalytic subset of dimension $\leq d$ of $\R^N$ for some~$N\geq 0$.  Since~$X_i$ is homeomorphic to a countable simplicial complex of dimension $\leq d$ (see \cite[Theorem~2]{Lojasiewicz}), its covering dimension (in the sense of \cite[Definition~1.6.7]{Engelking}) is~$\leq d$ (combine 
\cite[Theorems~1.8.2, 3.1.4 and 3.1.8]{Engelking}).  We deduce that $X/\Gamma$ also has covering dimension $\leq d$ (use  \cite[Theorems 1.7.7 and~1.12.10]{Engelking}).
The lemma now follows from \cite[II, Corollary 16.34]{Bredon} (note that the notions of covering dimension based on finite or possibly infinite coverings coincide by \cite[Proposition 3.2.2]{Engelking}).
\end{proof}

\begin{thm}
\label{thArtin}
Let $S$ be a $G$-equivariant Stein space of dimension $n$. Let $K\subset S$ be a $G$-invariant~$\cO(S)$\nobreakdash-convex compact subset. Let $\F$ be a $G$\nobreakdash-equivariant weakly constructible sheaf on $S$. Then $H^k_G(S,S^G\cup K,\F)=0$ for $k>n$.
\end{thm}

\begin{proof}
Denote by $j:S\setminus (S^G\cup K)\to S$ the inclusion map. 
Let $j_!j^*\F$ denote the extension by zero of the sheaf $j^*\F$.
Then one has $H^k(S,S^G\cup K,\F)=H^k(S,j_!j^*\F)$ and $H^k_G(S,S^G\cup K,\F)=H^k_G(S,j_!j^*\F)$ for all $k\geq 0$.

Let $\Z(1)$ be the $G$-module isomorphic to $\Z$ as an abelian group, on which the complex conjugation $\sigma\in G$ acts by multiplication by $-1$.  Set $\F(1):=\F\otimes_{\Z}\Z(1)$.
For~$k>n$,  the group $H^k(S,S^G\cup K,\F)$ vanishes by Proposition \ref{propArtin2}.
It therefore follows from the exact sequences 
\begin{alignat}{4}
\label{rc1}&\dots\to 
H^k(S,j_!j^*\F)\to H^k_G(S,j_!j^*\F(1))\to H_G^{k+1}(S,j_!j^*\F)\to\dots&
\textrm{ \hspace{.3em}and }
\\
\label{rc2}&\dots\to 
 H^k(S,j_!j^*\F)\to H_G^k(S,j_!j^*\F)\to H_G^{k+1}(S,j_!j^*\F(1))\to\dots&
\end{alignat}
(for which see \cite[(1.6) and (1.7)]{BW1})
that the cohomology group~$H^k_G(S,S^G\cup K,\F)$ injects into~$H^{k+2}_G(S,S^G\cup K,\F)$, hence into~$H_G^{k+2i}(S,S^G\cup K,\F)$ for all $i\geq 0$. 

Let $\pi:S\to S/G$ be the quotient map.  Consider the sheaf $\G:=(\pi_*j_!j^*\F)^G$ on~$S/G$. 
Then $H_G^{k+2i}(S,S^G\cup K,\F)=H^{k+2i}_G(S,j_!j^*\F)=H^{k+2i}(S/G,\G)$
(the second equality follows from the first spectral sequence of \cite[Th\'eor\`eme 5.2.1]{Tohoku} since the stalks of $j_!j^*\F$ along $S^G$ vanish).
This group vanishes when $i\gg 0$ by Lemma \ref{lemvandim} applied with $X=S$ and $\Gamma=G$. We deduce that $H^k_G(S,S^G\cup K,\F)=0$, which is what we wanted to prove.
\end{proof}

\begin{rems}
\label{remArtinG}
(i) Proposition \ref{propArtin2} can be deduced from Theorem \ref{thArtin} by applying it to the weakly constructible $G$-equivariant sheaf $\F\otimes_{\Z}\Z[G]$.  Similarly,  one can recover Proposition \ref{propArtin} from Theorem \ref{thArtin} by applying it to the $G$-equivariant Stein space~$S\sqcup S^{\sigma}$ where $S^{\sigma}$ is the conjugate complex space of $S$ (see \eg \cite[\S A.2]{Tight}) and where the action of $G$ exchanges $S$ and $S^{\sigma}$.

(ii) Theorem \ref{thArtin} holds more generally for $\F\in\leftindex^p{D}_{G,\wc}^{\leq n}(S)$,
where the triangulated category~${D}^b_{G,\wc}(S)$ of weakly constructible bounded complexes of $G$-equivariant sheaves on~$S$ is endowed with its perverse $t$-structure for the middle perversity. To see it, note that \cite[Theorem~3.64~(1)]{Constructible} is stated in this generality and that all the other arguments used in the proof extend at once from ($G$-equivariant) sheaves to bounded complexes of ($G$-equivariant) sheaves. Analogous remarks apply to Propositions \ref{propArtin} and \ref{propArtin2}.
\end{rems}

\subsection{$G$-equivariant vanishing on alterations}

We conclude this section by presenting a consequence of Theorem \ref{thArtin} which we will use in \cite{Steinsurface}. 

\begin{prop}
\label{lemvanmodif}
Let $S$ be a $G$-equivariant Stein space of dimension $n$.  Let ${p:T\to S}$ be a proper $G$-equivariant holomorphic map. Assume that there exists a nowhere dense closed analytic subset $\Sigma\subset S$ such that $p^{-1}(\Sigma)$ is nowhere dense in $T$ and ${p|_{p^{-1}(S\setminus \Sigma)}:p^{-1}(S\setminus \Sigma)\to S\setminus \Sigma}$ is finite.  Let $\F$ be a $G$-equivariant weakly constructible sheaf on $T$. Then $H^k_G(T,T^G,\F)=0$ for $k>\max(n,2n-2)$.
\end{prop}

\begin{proof}
Let $j:p^{-1}(S^G)\setminus T^G\to p^{-1}(S^G)$ and $i:p^{-1}(S^G)\to T$ be the inclusion maps
and let $\pi:p^{-1}(S^G)\to p^{-1}(S^G)/G$ be the quotient map. Then
\begin{equation}
\label{relcohos}
H^k_G(p^{-1}(S^G),T^G,\F)=H^k_G(p^{-1}(S^G), j_!j^*i^*\F)=H^k(p^{-1}(S^G)/G,\G),
\end{equation}
where $\G:=(\pi_*j_!j^*i^*\F)^G$ (the second equality follows from the first spectral sequence of \cite[Th\'eor\`eme 5.2.1]{Tohoku} since the stalks of $j_!j^*i^*\F$ along $T^G$ vanish).
Our hypotheses on $p$ imply that the analytic subset of $S$ over which the fibers of $p$ have (complex) dimension $\geq d$ has dimension $\leq n-1-d$ if ${d>0}$. It follows that the real-analytic subset $p^{-1}(S^G)$ of $T$ has (real) dimension $\leq \max(n,2n-2)$.  
We deduce from Lemma \ref{lemvandim} that the group (\ref{relcohos}) vanishes if~$k>\max(n,2n-2)$.  

Let $u:T\setminus (p^{-1}(S^G))\to T$ and $v:S\setminus S^G\to S$ be the inclusion maps. It follows from proper base change \cite[III, Theorem 6.2]{Iversen} that $\RR^sp_*(u_!u^*\F)=v_!v^*\RR^sp_*\F$.
The Leray spectral sequence of $p$ for the $G$-equivariant sheaf $u_!u^*\F$ therefore reads
\begin{equation}
\label{HS}
E_2^{r,s}=H^r_G(S,S^G,\RR^sp_*\F)\implies H^{r+s}_G(T,p^{-1}(S^G), \F).
\end{equation}

The sheaves $\RR^sp_*\F$ are weakly constructible (see \cite[Theorem 2.5]{Constructible}).  The support of $\RR^sp_*\F$ is included in the locus where $p$ has fibers of dimension $\geq \frac{s}{2}$ (combine proper base change  \cite[III, Theorem 6.2]{Iversen} and Lemma \ref{lemvandim}), and hence has dimension $\leq n-1-\frac{s}{2}$ if~$s>0$.  It therefore follows from (\ref{HS}) and Theorem~\ref{thArtin}
that $H^k_G(T,p^{-1}(S^G), \F)=0$ if $k>\max(n,2n-2)$.  As (\ref{relcohos}) vanishes in the same range, the long exact sequence of $G$-equivariant cohomology of the triple $(T,p^{-1}(S^G),T^G)$ concludes the proof.
\end{proof}

\bibliographystyle{myamsalpha}
\bibliography{Artinvanishing}

\end{document}